\documentclass[11pt, reqno]{amsart}

\usepackage{amsmath, amssymb, amsthm, amsfonts, graphicx, color, enumerate, eucal, fullpage}
\usepackage[utf8]{inputenc}
\definecolor{limegreen}{rgb}{0.196,0.804,0.196}
\definecolor{darkgreen}{rgb}{0.0,0.5,0.0}
\definecolor{darkbluegreen}{rgb}{0,0.3,0.6}
\definecolor{badgerred}{rgb}{0.715,0.004,0.004}
\usepackage[font=small,labelfont={bf,sf}, textfont={sf},margin=0em]{caption}
\usepackage{url,hyperref}
\hypersetup{colorlinks, 
citecolor=badgerred,
filecolor=black,
linkcolor=blue,
urlcolor=darkgreen,
bookmarksopen=false,
pdftex}

\newtheorem{theorem}{Theorem}[section]
\newtheorem{lemma}[theorem]{Lemma}
\newtheorem{definition}[theorem]{Definition}
\newtheorem{prop}[theorem]{Proposition}
\newtheorem{corollary}[theorem]{Corollary}

\numberwithin{equation}{section}

\begin{document}

\title{On the rate of convergence of the rescaled mean curvature flow} \author{Rory Martin-Hagemayer} \author{Natasa Sesum} 
\maketitle
\begin{abstract}

We estimate from above the rate at which a solution to the rescaled mean curvature flow on a closed hypersurface may converge to a limit self-similar solution, i.e. a shrinker. Our main result implies that any solution which converges to a shrinker faster than any fixed exponential rate must itself be shrinker itself.

\end{abstract}
\begingroup\footnotesize\sffamily\tableofcontents \endgroup

\section{Introduction}\label{sec-intro}

 We say that a map $\bar\varphi:M\times[0,T)\to\mathbb{R}^{n+1}$ is a \emph{mean curvature flow} starting from an embedding $\varphi_0:M\to\mathbb{R}^{n+1}$ provided
\begin{equation}\label{eq-mcf}\begin{cases}
    \frac{\partial\bar\varphi}{\partial \tau}(p,\tau)=H(p,\tau)\nu(p,\tau)
\\
\bar \varphi(p,0)=\bar \varphi_0(p)

\end{cases}\end{equation}
where $H(p,\tau)$ is the mean curvature at $\bar \varphi(p,\tau)$ and $\nu(p,\tau)$ is the normal vector at $\bar\varphi(p,\tau)$. The mean curvature flow can be extended in time so long as the norm of second fundamental form $|A|^2$ remains bounded. In other words, the mean curvature flow develops a singularity precisely if $|A|^2$ approaches infinity, developing a singularity at some time $T$. There are several ways to rescale mean curvature flow around a singularity. Given a mean curvature flow $\tilde{\varphi}:M\times [-1, 0)\to \mathbb{R}^{n+1}$, one can rescale by the equations 
$$\tilde \varphi(\cdot,t)=\frac{1}{\sqrt{-\tau}}\bar{\varphi}(\cdot,\tau)$$
where $t=-ln(-\tau)$. One can verify that this rescaled mean curvature flow moves by the equation in the following definition.
\begin{definition} We say that $\tilde\varphi:M\times[0,\infty)\to\mathbb{R}^{n+1}$ is a \emph{rescaled mean curvature flow} provided it moves by the equation
\begin{equation}
\label{eq-rescaled}
    \frac{\partial}{\partial t}\tilde\varphi(p,t)=H(p,t)\nu(p,t)+\frac{\tilde\varphi(p,t)}{2}
\end{equation}
where $H(\cdot,t)$ is the mean curvature at $\tilde\varphi(\cdot,t)$ measured after rescaling.
\end{definition}

If $M_t$ is a family of hypersurfaces moving by \eqref{eq-rescaled}, and if as $t\to \infty$ we have that the $\lim_{t\to\infty} M_t = \Sigma$, for some compact nonempty surface $\Sigma$, then Huisken's monotinicity formula implies that $\Sigma$ is a shrinker, satisfying 
\[
    H+\frac{x\cdot\nu}{2}\equiv 0.
\]
In this paper we prove the following theorem.

\begin{theorem}
\label{thm-main}
Let $\widetilde{\varphi}:M\times[0,\infty)\rightarrow \mathbb{R}^{n+1}$ be a solution to the rescaled mean curvature flow,
\begin{equation}
\label{eq-mcf-rescaled}\frac{\partial \widetilde{\varphi}}{\partial t}=\vec{H}+\frac{x}{2},
\end{equation}
such that we can find some $\tau_1,\tau_2\in[-1,0)$ ($\tau_1<\tau_2$) such that the corresponding (non-rescaled) mean curvature flow $\varphi:M\times[\tau_1,\tau_2)\to\mathbb{R}^{n+k}$ is sufficiently close in measure to the homothetically shrinking flow of a (closed) shrinker $\Sigma$.
Then there exists $T>0$ such that on $[T,\infty)$ we may express $M_t$ as a normal graph of a function $v(p,t)$ over $\Sigma$. Furthermore one of the following must hold.
\begin{enumerate}
    \item There exist $C,m>0$ such that
$
\|v(\cdot,t_i)\|_{C^{2,\alpha}(\Sigma)}\geq Ce^{-mt_i}
$.
\item $M_t$ moves by diffeomorphisms and is identically $\Sigma$.

\end{enumerate}\end{theorem}
This theorem is analogous to results achieved by Sun and Wang for the Kähler-Ricci flow \cite{SW} and Kotschwar for the Ricci flow \cite{Ko}. In the former paper Sun and Wang proved that a Kähler-Ricci flow starting in a sufficiently small $C^{k,\gamma}$ neighborhood of a Kähler-Einstein manifold of Einstein constant 1, then it must converge at a polynomial rate under their normalized Kähler-Ricci flow to another Kähler-Einstein metric. In \cite{Ko} Kotschwar showed that if $(M,\Bar{g})$ is a closed Ricci soliton and $g(t)$ is a smooth solution to the normalized Ricci flow such that up to a sequence of diffeomorphisms $\psi_i\in\text{Diff}(M)$ we have $\psi_i^*g(t_i)\to\Bar{g}$ in $C^k_{\Bar{g}}$ for all k and for some sequence of times $t_i\to\infty$, then either
$\|\psi_ig(t_i)-\Bar{g}\|_{C^2_{\Bar{g}}}\geq Ce^{-mt_i}$
for some $m,C>0$ or there is a smooth family of diffeomorphisms $\Phi_t\in\text{Diff}(M)$ such that $\Phi_t^*g(t)=\Bar{g}$ for all $t>0$.

The approach we take here is quite similar to that of Kotschwar, to whom our result is more visibly similar. We use the work of Kotschwar, specifically his energy functional approach, to draw insights into how to set up our proof of this result.

We follow the general outline of the proof by contradiction used by Kotschwar in \cite{Ko} in the setting of Ricci flow. 
First we prove that we may express the flow $M_t$ as a normal graph over the shrinker $\Sigma$ past some time $T_0$. We use a theorem of Felix Schultze in \cite{Sc} to establish the needed graphicality. It allows us to use the Lojasiewicz inequality for mean curvature flow found in \cite{Sc}, which comes from a graphical representation of the flow. We use this to establish needed bounds on energy functional $E(t)$ that we introduce later in the text, and that are used in the proof of Theorem \ref{thm-main}.

More precisely, assuming that Theorem \ref{thm-main} were false, we would have that for every $C,m>0$, we eventually have 
$$\|{v(\cdot,t_i)}\|_{C^{2,\alpha}(\Sigma)}< Ce^{-mt_i}$$
for $i$ large enough, for our family of hypersurfaces $M_t$ given by the graphs of $v(t)$ over $\Sigma$, with $M_t$ not being a shrinker. We then show that this implies  we may also find constants $C_m'$ such that for Huisken's functional $\Omega(M)=\int_Me^{-\frac{|x|^2}{4}}d\mu$ we have

\begin{equation}\label{eq:1}
\Omega(M_{t_i})-\Omega(\Sigma)\leq C^{'}_me^{-mt_i}.\end{equation}  

We actually choose $E(t)$ so that recalling Huisken's monotonicity formula we can write
\begin{equation*}
    \Omega(M_t)-\Omega(M_s)=\int_t^s E(\xi)d\xi.
\end{equation*}
If we had some sort of exponential control of $E(t)$, like
\begin{equation}\label{eq-E-below}
E(t)\geq E(T_0)e^{-K'(t-T_0)},
\end{equation}
this would allow us to write (pushing $s\to\infty$ in Huisken's monotonicity formula) that
\begin{equation*}
\Omega(M_t)-\Omega(\Sigma)\geq \frac{E(T_0)}{K'}e^{-K'(t-T_0)},
\end{equation*}
which we could combine with \eqref{eq:1} to write that
$$E(T_0)\leq Ce^{(K'-m)(t_i-T_0)}.$$
If we had such an inequality holding for all $i$ and all $m$ large, we could conclude that $E(T_0)=0$, and thus that $M_{T_0}$ were a shrinker. In \cite{HH}, Hong Huang proved that the mean curvature flow satisfies a backwards uniqueness property, so if $M_{T_0}$ were a shrinker, our flow would have to be a shrinker for all times, which would contradict our original assumption.

Above discussion implies that we devote our efforts precisely towards proving \eqref{eq-E-below}.
Inspired by the proof in \cite{Ko} our key tool for doing so will be analyzing the Dirichlet-Einstein quotient, which we define below.
\\
\begin{definition} Given a family of hypersurfaces moving by the normal rescaled mean curvatire flow \eqref{eq-normal-mcf} or the recaled mean curvature flow (2) (which is not a shrinker), we define the \emph{Dirichlet-Einstein quotient} $N(t)$ by the following formula
\begin{equation}
    N(t)=\frac{\Dot{E}(t)}{E(t)}=\frac{d}{dt}\text{ln}(E(t)).
\end{equation}
\end{definition}
\par
The usefulness of the Dirichlet-Einstein quotient comes from the fact that if we are able to bound the Dirichlet-Einstein quotient from below, that is, we are able to prove that we have some $K'>0$ such that 
$$
    N(t)\geq -K'.
$$
This then implies
$$\text{ln}(E(t))=\text{ln}(E(T_0)+\int_{T_0}^tN(s)ds\geq \text{ln}(E(T_0)-K'(t-T_0)$$
and from this we are able to establish \eqref{eq-E-below}.

The outline of the paper is as follows. In section \ref{sec-ev-eq} we derive evolution equation of quantities that we need to define our energy functionals and their evolution equations in later sections. In section \ref{sec-energy} we define the energy functional $E(t)$ which is strongly linked to Huisken's monotonicity formula and compute its first and second derivative in time. In section \ref{sec-dirichlet-quotient} we introduce Dirichlet-Einstein quotient $N(t)$, compute its rate of change in time and prove the lower bound on $\dot{N}(t)$ in terms of $S$ and its derivatives. 

\section{Geometric Quantities and Evolution Equations}
\label{sec-ev-eq}
To derive certain estimates, we will sometime consider  the so called \emph{normal rescaled mean curvature flow} \eqref{eq-normal-mcf}
\begin{equation}\label{eq-normal-mcf} \frac{\partial {\varphi}}{\partial t}=\vec{H}+\frac{x^\perp}{2},\end{equation}
which we show is equivalent to \eqref{eq-mcf} up to tangential diffeomorphisms. 
In order to make our computations simpler we will first understand how quantities evolve on a hypersurface moving by \eqref{eq-normal-mcf}. For convenience we define the quantity $S=H+\frac{x\cdot\nu}{2}$ and rewite the flow \eqref{eq-normal-mcf} in the form
$$\frac{\partial\varphi}{\partial t}=S\nu.$$
Later on, we will define our energy functional by the equation $$E(t)=\int_{M_t}S^2e^{-\frac{|x|^2}{4}}d\mu$$
and analyze its evolution in time to determine its behavior as $t\to\infty$.
In this section we compute the evolution equations of key geometric quantities which will be necessary later on to perform the calculations on our energy functional which will be necessary to prove theorem 1.3. To calculate the evolution equation for $E(t)$, we would write
$$\frac{\partial}{\partial t}E(t)=2\int_{M_t}S\frac{\partial}{\partial t}Se^{-\frac{|x|^2}{4}}d\mu+\int_{M_t}S^2\frac{\partial}{\partial t}[e^{-\frac{|x|^2}{4}}]d\mu+\int_{M_t}S^2e^{-\frac{|x|^2}{4}}\frac{\partial}{\partial t}d\mu$$
so at the very least we will have to systematically work our way up to the calculation of $\partial_tS$, $\partial_te^{-\frac{|x|^2}{4}}$, and $\partial_td\mu$. We do so, calculating some other important evolution equations along the way, below.

\subsection{Preliminary Calculations}

In this subsection we prove the following key lemma about the evolution of geometric quantities which are necessary for our computations on the energy functional later. The calculations are done for quantities while evolving by \eqref{eq-normal-mcf}. The computations are standard, so we will omit the details.
\begin{lemma}\label{lemma-evolutions}
The following calculations hold true on any family of hypersurfaces $M_t\subset\mathbb{R}^{n+1}$ moving by \eqref{eq-normal-mcf}
 \begin{equation}\label{eq-g}
        \frac{\partial}{\partial t}g_{ij}=-2A_{ij}S
    \end{equation}
\begin{equation}\label{eq-g-inv}
        \partial_tg^{jk}=2A^{jk}S
    \end{equation}
 \begin{equation}
 \label{eq-Gamma}
        \frac{\partial}{\partial t}\Gamma^k_{ij}=-g^{kl}\Big{[}\frac{\partial}{\partial x_j}\big{(}A_{il}S\big{)}+\frac{\partial}{\partial x_i}\big{(}A_{jl}S\big{)}-\frac{\partial}{\partial x_l}\big{(}A_{ij}S\big{)}\Big{]}+SA^{kl}\Big{[}\frac{\partial}{\partial x_j}g_{il}+\frac{\partial}{\partial x_i}g_{jl}-\frac{\partial}{\partial x_l}g_{ij}\Big{]}
    \end{equation}
\begin{equation}\label{eq-nu}
    \frac{\partial\nu}{\partial t}=-\nabla S.
\end{equation}
\begin{equation}\label{eq-Aij}
    \frac{\partial}{\partial t}A_{ij}=\nabla_i\nabla_jS-SA_{jl}A_{il}.
\end{equation}

\end{lemma}
\emph{Proof:} In this proof we use Mantegazza's computations of the evolution of basic geometric quantities in sections 1.2 and 2.3 of \cite{Ma}, but do so keeping track of the fact that for us $\frac{\partial\varphi}{\partial t}=S\nu$, rather than $H\nu$. One may use the Gauss-Weingarten relations (see e.g. \cite{Ma} pg 4)
$$\begin{cases}
\frac{\partial^2\varphi}{\partial x_i\partial x_j}=\Gamma^k_{ij}\frac{\partial\varphi}{\partial x_k}+A_{ij}\nu\\
\frac{\partial}{\partial x_j}\nu=-A_{jl}g^{ls}\frac{\partial\varphi}{\partial x_s}
\end{cases}$$
and exponential coordinates at a point to compute the rate of change of the metric $g_{ij}$ on $M_t$ by using the formula of the first fundamental form as follows
\begin{equation*}
    \begin{split}
        \frac{\partial}{\partial t}g_{ij}=&\frac{\partial}{\partial t}\langle\frac{\partial\varphi}{\partial x_i},\frac{\partial\varphi}{\partial x_j}\rangle\\
        =&\frac{\partial}{\partial x_i}\langle\frac{\partial\varphi}{\partial t}^T,\frac{\partial\varphi}{\partial x_j}\rangle+\frac{\partial}{\partial x_j}\langle\frac{\partial\varphi}{\partial t}^T,\frac{\partial\varphi}{\partial x_i}\rangle-2\Gamma^k_{ij}\langle\frac{\partial\varphi}{\partial t}^T,\frac{\partial\varphi}{\partial x_k}\rangle -2A_{ij}\langle\frac{\partial\varphi}{\partial t},\nu\rangle\\
        =&-2A_{ij}S,
    \end{split}
\end{equation*}
which gives \eqref{eq-g}.

One may use the fact that $g_{ij}g^{jk}=\delta_i^k$ and \eqref{eq-g} to calculate that 
$$\partial_tg^{jk}=2A^{jk}S,$$
which is \eqref{eq-g-inv}.

Combining \eqref{eq-g}, \eqref{eq-g-inv} and \eqref{eq-normal-mcf} easily yields \eqref{eq-Gamma}.
%

Straightforward computation shows that
$\langle\frac{\partial\nu}{\partial t},\frac{\partial\varphi}{\partial x_i}\rangle=-\langle\nu,\frac{\partial^2\varphi}{\partial t\partial x_i}\rangle=-\langle\nu,\frac{\partial(S\nu)}{\partial x_i}\rangle=-\frac{\partial S}{\partial x_i}$ which can be expressed more succinctly as 
\begin{equation*}
    \frac{\partial\nu}{\partial t}=-\nabla S,
\end{equation*}
which gives \eqref{eq-nu}.

Using the formula for the second fundamental form, the Gauss-Weingarten relations, we perform the following calculation to determine the rate of change of entries of the second fundamental form under \eqref{eq-normal-mcf}.
\begin{equation*}
    \begin{split}
        \frac{\partial}{\partial t}A_{ij}=&\frac{\partial}{\partial t}\langle \nu, \frac{\partial^2\varphi}{\partial x_i\partial x_j}\rangle\\
        =&\langle\nu,\frac{\partial^2(S\nu)}{\partial x_i\partial x_j}\rangle-\langle\nabla S, \frac{\partial^2\varphi}{\partial x_i\partial x_j}\rangle\\
        =&\frac{\partial^2S}{\partial x_i\partial x_j}-S\langle\nu,\frac{\partial}{\partial x_i}\big{(}A_{jl}\frac{\partial\varphi}{\partial x_l}\big{)}\rangle-\langle\frac{\partial S}{\partial x_l}\frac{\partial\varphi}{\partial x_s}g^{ls},\Gamma^k_{ij}\frac{\partial\varphi}{\partial x_k}+A_{ij}\nu\rangle\\
        =&\frac{\partial^2S}{\partial x_i\partial x_j}-SA_{jl}A_{il}-\frac{\partial S}{\partial x_l}g^{ls}g_{sk}\Gamma^k_{ij}\\
        =&\nabla_i\nabla_jS-SA_{jl}A_{il},
    \end{split}
\end{equation*}
which gives \eqref{eq-Aij}. $\blacksquare$
\newline

\subsection{The Differential Operator $\mathcal{L}$}
Colding and Minicozzi introduced the differential operator 
\begin{equation*}
    \mathcal{L}=\Delta-\frac{1}{2}\langle x,\nabla\cdot\rangle
\end{equation*}
which satisfies the condition that
$$\int_M f(\mathcal{L}g)e^{-\frac{|x|^2}{4}}d\mu=\int_M (\mathcal{L}f) ge^{-\frac{|x|^2}{4}}d\mu=-\int_M\nabla f\nabla g e^{-\frac{|x|^2}{4}}d\mu$$
for every $f,g$ suitably differentiable. It is sometimes referred to as the ``stability operator" and will be vital during our calculations of the evolution of energy functionals as it shows up in many evolution equations and is self-adjoint.

Next we shall recall some well known formulas for shrinkers, using the equation $H=-\frac{x^\perp}{2}$, which we reframe for our setting where the codimension is one as follows
$$\mathcal{L}|A|^2=2|\nabla A|^2+2|A|^2(\frac{1}{2}-|A|^2)$$
$$\mathcal{L}H^2=H^2-2H^2|A|^2+2|\nabla H|^2=2|\nabla H|^2+2H^2(\frac{1}{2}-|A|^2),$$
see, for example Ding and Lin \cite{DL} equations (2.6), (2.8) for the more general form. 


We will attempt to produce similar fomulas but for general hypersurfaces rather than merely for shrinkers. We follow their derivation, but modify the procedure by starting instead from the equations $H=-\frac{x\cdot\nu}{2}+S$ and $\frac{x\cdot\nu}{2}=S-H$. We collect our results in the following lemma.
\begin{lemma}\label{lemma-L}
The following calculations hold true for any hypersurface $M\subset \mathbb{R}^{n+1}$

\begin{equation}\label{eq-Hij}\nabla_i\nabla_jH=\frac{1}{2}A_{ij}+(S-H)A_{ik}A_{jk}+\frac{1}{2}\langle x,e_k\rangle\nabla_iA_{jk}+\nabla_i\nabla_jS\end{equation}
\begin{equation}\label{eq-A2}\mathcal{L}|A|^2=2|\nabla A|^2+2|A|^2(\frac{1}{2}-|A|^2)+2S\text{tr}(A^3)+2A_{ij}\nabla_i\nabla_jS\end{equation}
\begin{equation}\label{eq-H}\mathcal{L}H=\frac{1}{2}H+(S-H)|A|^2+\Delta S.\end{equation}

\end{lemma}
\emph{Proof:} First observe that we can use the Gauss-Weingarten relations to perform covariant differentiation on $H=-\frac{x\cdot\nu}{2}+S$ to get
$$\nabla_jH=\frac{1}{2}\langle x,e_k\rangle A_{jk}+\nabla_jS.$$
Differentiating and using Gauss-Weingarten once more, we calulate
$$
\nabla_i\nabla_jH=\frac{1}{2}A_{ij}+\frac{1}{2}\langle x,\nu\rangle A_{ik}A_{jk}+\frac{1}{2}\langle x,e_k\rangle\nabla_iA_{jk}+\nabla_i\nabla_jS,$$
and then substituting in $\frac{x\cdot\nu}{2}=S-H$ gives
$$
\nabla_i\nabla_jH=\frac{1}{2}A_{ij}+(S-H)A_{ik}A_{jk}+\frac{1}{2}\langle x,e_k\rangle\nabla_iA_{jk}+\nabla_i\nabla_jS,$$
which is \eqref{eq-Hij}.

We also immediately get another useful equation to refer back to later
\begin{equation}
\nabla_i\nabla_jS=\nabla_i\nabla_jH-\frac{1}{2}A_{ij}-(S-H)A_{ik}A_{jk}-\frac{1}{2}\langle x,e_k\rangle\nabla_iA_{jk}
\end{equation}
from \eqref{eq-Hij}.

At this point, we refer to a  standard calculation (which holds for general hypersurfaces), see Ding and Lin \cite{DL} equation (2.2) for details, which in codimension $1$ says that
$$\Delta|A|^2=2|\nabla A|^2+2A_{ij}\nabla_i\nabla_jH+2H\text{tr}(A^3)-2|A|^4.$$
Substituting in for $\nabla_i\nabla_jH$ using \eqref{eq-Hij} gives
\begin{equation*}
    \begin{split}
        \mathcal{L}|A|^2=&\Delta |A|^2-\frac{1}{2}\langle\nabla |A|^2,x\rangle\\
        =&2|\nabla A|^2+2A_{ij}\Big{[}\frac{1}{2}A_{ij}+(S-H)A_{ik}A_{jk}+\frac{1}{2}\langle x,e_k\rangle\nabla_iA_{jk}+\nabla_i\nabla_jS\Big{]}+2H\text{tr}(A^3)-2|A|^4-\frac{1}{2}\langle\nabla |A|^2,x\rangle\\
        =&2|\nabla A|^2+2|A|^2(\frac{1}{2}-|A|^2)+2S\text{tr}(A^3)+2A_{ij}\nabla_i\nabla_jS
    \end{split}
\end{equation*}
and this shows \eqref{eq-A2}.

We can also contract (14) with $g^{ij}$ to get
\begin{equation*}
    \begin{split}
        \mathcal{L}H=&g^{ij}\Big{[}\frac{1}{2}A_{ij}+(S-H)A_{ik}A_{jk}+\frac{1}{2}\langle x,e_k\rangle\nabla_iA_{jk}+\nabla_i\nabla_jS\Big{]}-\frac{1}{2}\langle x,\nabla H\rangle\\
        =&\frac{1}{2}H+(S-H)|A|^2+\Delta S
    \end{split}
\end{equation*}
which is \eqref{eq-H}. $\blacksquare$
\newline 
\newline
\textbf{Remark}
In the next subsection, we will compute the evolution of some geoemtric quantities with the geometric relationships we have established in this subsection. The unique usefulness of $\mathcal{L}$ does not end here, however, as its self-adjointness will be vital to establish the needed control of the evolution of our energy functional.
\subsection{Evolution of Geometric Quantities}

In this subsection we will use our calculations up to this point to compute the evolution equations of key geometric quantities for a family of hypersurfaces moving by \eqref{eq-normal-mcf} in the form that will be most beneficial for us throughout this paper. Key for us will be to express as much of the quantity as possible in terms of $S$ and its covariant derivatives $\nabla^lS$, which will be most useful for us later on as these quantities all converge to $0$ as $t\to\infty$.
\subsubsection{The Normalizer Lemma}
We shall prove a quick preliminary lemma that we shall refer to as the \emph{normalizer lemma} which allows us to easily transition from time evolutions of scalar geometric quantities with respect to \eqref{eq-mcf}, to time evolutions of scalar geometric quantities with respect to \eqref{eq-normal-mcf}.
\begin{lemma} \textbf{(Normalizer Lemma)}\label{lemma-normal} If $\varphi(p,t)$ is a smooth family of embeddings of a hypersurface $M$ moving according to \eqref{eq-mcf}, and we have that for some function $f(p,t)$ which is defined on $M$ and we have that $\psi(p,t)$ is a smooth family of diffeomorphisms such that $\varphi(\psi(p,t),t)$ is a smooth family of embeddings moving by \eqref{eq-normal-mcf}, where
\begin{equation}
\label{eq-diffeo}
D_p\varphi(\psi(p,t),t)(\frac{\partial\psi}{\partial t}(p,t))=-\frac{\varphi(\psi(p,t),t)^T}{2},
\end{equation}
then
$$\frac{d}{dt}f(\psi(p,t),t)=\frac{\partial}{\partial t}f(\psi(p,t),t)-\frac{1}{2}\langle x,\nabla f\rangle.$$
\end{lemma}
\emph{Proof:} This is a straightforward computation that follows by  computing the total deriative of $f(\psi(p,t),t)$. Note that \eqref{eq-diffeo} defines diffeomorphisms up to which the flows \eqref{eq-mcf} and \eqref{eq-normal-mcf} are equivalent.$\blacksquare$\\

\subsubsection{Evolution Equations}
We are now ready to prove the following lemma. 

\begin{lemma}\label{lemma-evolutions1}
The following calculations hold true for any family of hypersurfaces $M_t\subset\mathbb{R}^{n+1}$ moving by \eqref{eq-normal-mcf}
\begin{equation}\label{eq-Aij-ev}
        \frac{\partial}{\partial t}A_{ij}=\mathcal{L}A_{ij}+(|A|^2-\frac{1}{2})A_{ij}-2SA_{ik}A_{jk}
    \end{equation}
\begin{equation}\label{eq-A2-ev}
        \frac{d}{dt}|A|^2=2A_{ij}\nabla_i\nabla_jS+2Str(A^3)
    \end{equation} 
\begin{equation}\label{eq-S-ev}
        \frac{d}{dt}S=\mathcal{L}S+(|A|^2+\frac{1}{2})S
    \end{equation}
 \begin{equation}\label{eq-weight-ev}
        \frac{\partial}{\partial t}e^{-\frac{|x|^2}{4}}d\mu=-S^2e^{-\frac{|x|^2}{4}}d\mu
    \end{equation}
\end{lemma}
\emph{Proof:} Equation \eqref{eq-Aij} can be combined with \eqref{eq-A2} to calculate that 
$$\frac{\partial}{\partial t}A_{ij}=\nabla_i\nabla_jH-\frac{1}{2}A_{ij}-(S-H)A_{ik}A_{jk}-\frac{1}{2}\langle x,e_k\rangle\nabla_iA_{jk}-SA_{jl}A_{il}$$
which, after we use Simon's Identity 
$$\nabla_i\nabla_jH=\Delta A_{ij}+|A|^2A_{ij}-HA_{il}A_{lj},$$
yields
\begin{equation*}
    \begin{split}
        \frac{\partial}{\partial t}A_{ij}=&\Delta A_{ij}+(|A|^2-\frac{1}{2})A_{ij}-\frac{1}{2}\langle x,e_k\rangle\nabla_kA_{ij}-2SA_{ik}A_{jk}\\
        =&\mathcal{L}A_{ij}+(|A|^2-\frac{1}{2})A_{ij}-2SA_{ik}A_{jk},
    \end{split}
\end{equation*}
which is \eqref{eq-Aij-ev}.

One may refer to standard calculations for the rescaled mean curvature flow which give that for a family of hypersurfaces moving by \eqref{eq-mcf} we have
$$\frac{d}{dt}|A|^2=-|A|^2+\Delta|A|^2-2|\nabla A|^2+2|A|^4$$
see e.g. \cite{Ma}, and we can combine this with \eqref{eq-A2} and Simon's identity to calculate that 
\begin{equation*}
    \begin{split}2A_{ij}\nabla_i\nabla_jS=&2A_{ij}\nabla_i\nabla_jH-|A|^2-2(S-H)tr(A^3)-\frac{1}{2}\langle x,\nabla|A|^2\rangle\\
    =&2A_{ij}\Delta A_{ij}+2|A|^4-|A|^2-2Str(A^3)-\frac{1}{2}\langle x,\nabla|A|^2\rangle\\
    =&\frac{d}{dt}|A|^2-2Str(A^3)-\frac{1}{2}\langle x,\nabla|A|^2\rangle,
    \end{split}
\end{equation*}
where in \emph{this above} equation $\frac{d}{dt}|A|^2$ is performed with respect to the rescaled mean curvature for \eqref{eq-mcf-rescaled} as opposed to the normal rescaled mean curvature flow \eqref{eq-normal-mcf}. We apply the normalizer lemma to this equation and calculate that for a family of hypersurfaces moving by \eqref{eq-normal-mcf}
\begin{equation*}
    \begin{split}
        \frac{d}{dt}|A|^2=&2A_{ij}\nabla_i\nabla_jS+2Str(A^3)
    \end{split}
\end{equation*}
which is \eqref{eq-A2-ev}.

One may again refer to standard calculations for the rescaled mean curvature flow which give that for a family of hypersurfaces moving by \eqref{eq-mcf-rescaled} we have
$$\frac{d}{dt}H=-\frac{1}{2}H+\Delta H+H|A|^2$$
see e.g. \cite{Ma}. We combine this with (11) and Simon's identity to calculate that
\begin{equation*}
    \begin{split}
        \Delta S=&g^{ij}\nabla_i\nabla_jS\\
        =&g^{ij}\Big{[}\nabla_i\nabla_jH-\frac{1}{2}A_{ij}-(S-H)A_{ik}A_{jk}-\frac{1}{2}\langle x,e_k\rangle\nabla_iA_{jk}\Big{]}\\
        =&-\frac{1}{2}H-(S-H)|A|^2-\frac{1}{2}\langle x, \nabla H\rangle\\
        &+g^{ij}\Big{[}\Delta A_{ij}+|A|^2A_{ij}-HA_{il}A_{lj}\Big{]}\\
        =&-\frac{1}{2}H-(S-H)|A|^2-\frac{1}{2}\langle x, \nabla H\rangle\\
        &+\Delta H+|A|^2H-H|A|^2\\
        =&\frac{d}{dt}H-\frac{1}{2}\langle x, \nabla H\rangle-S|A|^2,
    \end{split}
\end{equation*}
where here again we mean $\frac{d}{dt}H$ with respect to the rescaled mean curvature flow \eqref{eq-mcf-rescaled}. We apply Lemma \ref{lemma-normal} to this equation to calculate that, for a family of hypersurfaces moving by \eqref{eq-normal-mcf}, we have
\begin{equation*}
\begin{split}
\frac{d}{dt}H=&\Delta S+|A|^2S.
\end{split}
\end{equation*}
\eqref{eq-normal-mcf} and \eqref{eq-nu} combined give
\begin{equation*}
    \frac{\partial}{\partial t}\frac{x\cdot\nu}{2}=\frac{S}{2}-\langle\frac{x}{2},\nabla S\rangle.
\end{equation*}
Combining our last two calculations gives
\begin{equation*}
    \begin{split}
        \frac{\partial}{\partial t}S=&\frac{\partial}{\partial t}[H+\frac{x\cdot\nu}{2}]\\
        =&\Delta S+|A|^2S+\frac{S}{2}-\langle\frac{x}{2},\nabla S\rangle\\
        =&\mathcal{L}S+(|A|^2+\frac{1}{2})S,
    \end{split}
\end{equation*}
which is \eqref{eq-S-ev}.

We calculate, using a standard formula for the rate of change of the measure $d\mu$, that
\begin{equation*}
    \frac{d}{dt}d\mu=\frac{1}{2}g^{ij}\frac{d}{dt}g_{ij}d\mu=-SHd\mu.
\end{equation*}
This can be used to calculate that 
\begin{equation*}
\begin{split}
\frac{\partial}{\partial t}e^{-\frac{|x|^2}{4}}d\mu=&(-SH)e^{-\frac{|x|^2}{4}}d\mu+\frac{\partial}{\partial t}[e^{-\frac{|x|^2}{4}}]d\mu\\
=&(-SH-\frac{1}{4}\frac{\partial}{\partial t}[\big{<}x,x\big{>}])e^{-\frac{|x|^2}{4}}d\mu\\
=&(-SH-\frac{1}{2}\big{<}H\nu+\frac{x^\perp}{2},x\big{>})e^{-\frac{|x|^2}{4}}d\mu\\
=&(-SH-\frac{1}{2}S\big{<}\nu,x\big{>})e^{-\frac{|x|^2}{4}}d\mu\\
=&-S^2e^{-\frac{|x|^2}{4}}d\mu,
\end{split}
\end{equation*}
which is \eqref{eq-weight-ev}. $\blacksquare$

\subsection{Second Derivatives}
The final thing we will need to move on is a second derivative of $S$, as this will be necessary to compute the second derivative of our energy functional $E(t)$ and this will be necessary for the analysis of our energy functional.
\begin{lemma}
We have for a family of hypersurfaces $M_t\subset\mathbb{R}^{n+1}$ moving by \eqref{eq-normal-mcf} that
\begin{equation}\label{eq-dt2-S}
    \begin{split}
        \frac{d^2}{dt^2}S=&2SA_{ij}\nabla_i\nabla_jS+2A^{jk}\nabla_jS\nabla_kS-H|\nabla S|^2+Sg^{kl}\nabla_lS\nabla_kH\\
        &+S(2A_{ij}\nabla_i\nabla_jS+2S\text{tr}(A^3))\\
        &+\mathcal{L}\Big{[}\mathcal{L}S+(|A|^2+\frac{1}{2})S\Big{]}+(|A|^2+\frac{1}{2})\Big{[}\mathcal{L}S+(|A|^2+\frac{1}{2})S\Big{]}
    \end{split}
\end{equation}
\end{lemma}
\emph{Proof:} We saw in the previous subsection, equation \eqref{eq-S-ev}, that
$$\frac{d}{dt}S=\mathcal{L}S+(|A|^2+\frac{1}{2})S,$$
so to compute $\frac{d^2}{dt^2}S$, we have all we need with the exception of $\frac{d}{dt}\Delta S$. As such, we calculate that
\begin{equation*}
    \begin{split}
        \frac{\partial}{\partial t}\Delta S=&\frac{\partial}{\partial t}\Big{[}g^{ij}\nabla_i\nabla_j S\Big{]}\\
        =&\frac{\partial}{\partial t}\Big{[}g^{ij}\big{(}\frac{\partial}{\partial x_i}\frac{\partial}{\partial x_j}S-\Gamma^k_{ij}\nabla_kS\big{)}\Big{]}\\
        =&\frac{\partial}{\partial t}g^{ij}\big{(}\frac{\partial}{\partial x_i}\frac{\partial}{\partial x_j}S-\Gamma^k_{ij}\nabla_kS\big{)}\\
        &+\Delta\frac{\partial}{\partial t}S\\
        &-g^{ij}\frac{\partial}{\partial t}\Gamma^k_{ij}\nabla_k S.
    \end{split}
\end{equation*}
Using Lemma \ref{lemma-evolutions}, equations \eqref{eq-g-inv}, \eqref{eq-S-ev}, Simon's identity, Codazzi's equation, and normal coordinates at a point we compute,
\begin{equation}\label{eq-ev-lap-S}
    \begin{split}
        \frac{\partial}{\partial t}\Delta S=&\frac{\partial}{\partial t}g^{ij}\big{(}\nabla_i\nabla_jS\big{)}+\Delta\frac{\partial}{\partial t}S-g^{ij}\frac{\partial}{\partial t}\Gamma^k_{ij}\nabla_k S\\
        =&2SA^{ij}\nabla_i\nabla_jS\\
        &+\Delta\Big{[}\mathcal{L}S+(|A|^2+\frac{1}{2})S\Big{]}\\
        &+g^{ij}g^{kl}\Big{[}\frac{\partial}{\partial x_j}\big{(}A_{il}S\big{)}+\frac{\partial}{\partial x_i}\big{(}A_{jl}S\big{)}-\frac{\partial}{\partial x_l}\big{(}A_{ij}S\big{)}\Big{]}\nabla_kS\\
        =&2SA_{ij}\nabla_i\nabla_jS\\
        &+\Delta\Big{[}\mathcal{L}S+(|A|^2+\frac{1}{2})S\Big{]}\\
        &+A^{jk}\nabla_jS\nabla_kS+A^{ik}\nabla_iS\nabla_kS-H|\nabla S|^2\\
        &+2Sg^{ij}\nabla_jA_{il}\nabla_lS-Sg^{kl}\nabla_lS\nabla_kH\\
        =&2SA_{ij}\nabla_i\nabla_jS+2A^{jk}\nabla_jS\nabla_kS\\
        &-H|\nabla S|^2+Sg^{kl}\nabla_lS\nabla_kH\\
        &+\Delta\Big{[}\mathcal{L}S+(|A|^2+\frac{1}{2})S\Big{]}.\\
    \end{split}
\end{equation}
With \eqref{eq-ev-lap-S} computed, we can use it to plug in, calculating $\frac{d^2}{dt^2}S$ as follows

\begin{equation*}
    \begin{split}
        \frac{\partial}{\partial t}\Big{[}\mathcal{L}S+(|A|^2+\frac{1}{2})S\Big{]}=&\frac{\partial}{\partial t}\Delta S-\frac{1}{2}\frac{\partial}{\partial t}\langle x, \nabla S\rangle+S\frac{\partial}{\partial t}|A|^2+(|A|^2+\frac{1}{2})\frac{\partial}{\partial t}S\\
        =&2SA_{ij}\nabla_i\nabla_jS+2A^{jk}\nabla_jS\nabla_kS-H|\nabla S|^2+Sg^{kl}\nabla_lS\nabla_kH\\
        &+\Delta\Big{[}\mathcal{L}S+(|A|^2+\frac{1}{2})S\Big{]}-\frac{1}{2}\frac{\partial}{\partial t}\langle x, \nabla S\rangle\\
        &+S(2A_{ij}\nabla_i\nabla_jS+2S\text{tr}(A^3))\\
        &+(|A|^2+\frac{1}{2})\Big{[}\mathcal{L}S+(|A|^2+\frac{1}{2})S\Big{]}\\
        =&2SA_{ij}\nabla_i\nabla_jS+2A^{jk}\nabla_jS\nabla_kS-H|\nabla S|^2+Sg^{kl}\nabla_lS\nabla_kH\\
        &+S(2A_{ij}\nabla_i\nabla_jS+2S\text{tr}(A^3))\\
        &+\mathcal{L}\Big{[}\mathcal{L}S+(|A|^2+\frac{1}{2})S\Big{]}+(|A|^2+\frac{1}{2})\Big{[}\mathcal{L}S+(|A|^2+\frac{1}{2})S\Big{]}\\
    \end{split}
\end{equation*}
and this is \eqref{eq-dt2-S}. $\blacksquare$

\section{Energy Functional}\label{sec-energy}
In this section we introduce our energy functional and analyze it to control its long time behavior under the flow \eqref{eq-normal-mcf}.\\
\begin{definition} We define our energy functional $E(t)$ for a family of hypersurfaces $M_t\subset\mathbb{R}^{n+1}$ via the following formula
\begin{equation}
    E(t)=\int_{M_t}S^2e^{-\frac{|x|^2}{4}}d\mu.
\end{equation}
\end{definition}
\par One might recognize this as the negative right hand side of the rescaled Huisken's monotonicity formula (3), that is, we can write 
\begin{equation}
    \int_{M_t}e^{-\frac{|x|^2}{4}}d\mu_t-\int_{M_s}e^{-\frac{|x|^2}{4}}d\mu_s=\int_t^s E(\xi)d\xi,
\end{equation}
but it is, in a sense, physical as well, regarding $S$ as the ``speed" of \eqref{eq-normal-mcf} and $e^{-\frac{|x|^2}{4}}d\mu$ as the ``weight".

\subsection{Evolution of the Energy Functional}
In this section, we will continue to work under the assumption that $M_t$, or better yet $\varphi(M,t)$ is a smooth family of hypersurfaces evolving by the normal rescaled mean curvature flow \eqref{eq-normal-mcf}. In what follows we compute the evolution equation for $E(t)$. Much of the hard work has already been done for us in section 2.

As this quantity makes frequent appearences we will frequently refer to the weighted measure as
$$d\tilde{\mu}=e^{-\frac{|x|^2}{4}}d\mu$$
for cleanliness in our computations.
\begin{lemma}\label{lemma-energy-evolution}
For a family of hypersurfaces $M_t\subset\mathbb{R}^{n+1}$ moving by the normal rescaled mean curvature flow \eqref{eq-normal-mcf}, we have that
\begin{equation}\label{eq-energy-ev}
    \begin{split}
        \frac{\partial}{\partial t}E(t)=&-\int S^4d\tilde{\mu}+2\int S\Big{[}\mathcal{L}S+(|A|^2+\frac{1}{2})S\Big{]}d\tilde{\mu}
    \end{split}
\end{equation}
\end{lemma}
\emph{Proof:} We calculate directly from the definition of $E(t)$ that 
$$\frac{\partial}{\partial t}E(t)=2\int_{M_t}S\frac{\partial}{\partial t}Se^{-\frac{|x|^2}{4}}d\mu+\int_{M_t}S^2\frac{\partial}{\partial t}[e^{-\frac{|x|^2}{4}}d\mu]$$
and use \eqref{eq-S-ev} to plug in for $\frac{\partial}{\partial t}S$ and use \eqref{eq-weight-ev} to plug in for $\frac{\partial}{\partial t}[e^{-\frac{|x|^2}{4}}d\mu]$.

We obtain 
\begin{equation*}
    \begin{split}
        \frac{\partial}{\partial t}E(t)=&-\int S^4d\tilde{\mu}+2\int S\Big{[}\mathcal{L}S+(|A|^2+\frac{1}{2})S\Big{]}d\tilde{\mu}
    \end{split}
\end{equation*}
which gives \eqref{eq-energy-ev}. $\blacksquare$
\\
\newline We also would like to compute the second derivative of the energy functional, which we do in the following lemma.

\begin{lemma}\label{lemma-energy-der2}
For a family of hypersurfaces $M_t\subset\mathbb{R}^{n+1}$ moving by the normal rescaled mean curvature flow \eqref{eq-normal-mcf}, we have that
\begin{equation}\label{eq-energy-der2}
    \begin{split}
        \Ddot{E}(t)=&\int \Big{[}-S^3+2\Big{(}\mathcal{L}S+(|A|^2+\frac{1}{2})S\Big{)}\Big{]}^2d\tilde{\mu}\\
        &-2\int S^3\Big{[}\mathcal{L}S+(|A|^2+\frac{1}{2})S\Big{]}d\tilde{\mu}+2\int S\Big{[}4SA_{ij}\nabla_i\nabla_jS+Sg^{kl}\nabla_lS\nabla_kH+2S^2\text{tr}(A^3)\Big{]}d\tilde{\mu}\\
        &+2\int S\Big{[}2A^{jk}\nabla_jS\nabla_kS-H|\nabla S|^2\Big{]}d\tilde{\mu}
    \end{split}
\end{equation}
\end{lemma}
\emph{Proof:} We begin by reexamining the equation \eqref{eq-energy-ev}
\begin{equation*}
    \begin{split}
        \frac{\partial}{\partial t}E(t)=&-\int S^4d\tilde{\mu}+2\int S\Big{[}\mathcal{L}S+(|A|^2+\frac{1}{2})S\Big{]}d\tilde{\mu}
    \end{split}
\end{equation*}
and from this, and \eqref{eq-S-ev}, \eqref{eq-weight-ev}, we get, as a first expanded expression
\begin{equation*}
    \begin{split}
        \Ddot{E}(t)=&\int S^6d\tilde{\mu}-4\int S^3\Big{[}\mathcal{L}S+(|A|^2+\frac{1}{2})S\Big{]}d\tilde{\mu}\\
        &-2\int S^3\Big{[}\mathcal{L}S+(|A|^2+\frac{1}{2})S\Big{]}d\tilde{\mu}\\
        &+2\int \Big{[}\mathcal{L}S+(|A|^2+\frac{1}{2})S\Big{]}^2d\tilde{\mu}\\
        &+2\int S\frac{\partial}{\partial t}\Big{[}\mathcal{L}S+(|A|^2+\frac{1}{2})S\Big{]}d\tilde{\mu}.\\
        \end{split}
\end{equation*}
From here we recall equation \eqref{eq-dt2-S}, which we re-express as 
\begin{equation*}
    \begin{split}
        \frac{d}{dt}\Big{[}\mathcal{L}S+(|A|^2+\frac{1}{2})S\Big{]}=&2SA_{ij}\nabla_i\nabla_jS+2A^{jk}\nabla_jS\nabla_kS-H|\nabla S|^2+Sg^{kl}\nabla_lS\nabla_kH\\
        &+S(2A_{ij}\nabla_i\nabla_jS+2S\text{tr}(A^3))\\
        &+\mathcal{L}\Big{[}\mathcal{L}S+(|A|^2+\frac{1}{2})S\Big{]}+(|A|^2+\frac{1}{2})\Big{[}\mathcal{L}S+(|A|^2+\frac{1}{2})S\Big{]},
    \end{split}
\end{equation*}
and we will use \eqref{eq-dt2-S} to rewrite 
\begin{equation*}
    \begin{split}
        \int S\frac{\partial}{\partial t}\Big{[}\mathcal{L}S+(|A|^2+\frac{1}{2})S\Big{]}d\tilde{\mu}=&\int S\Big{[}\mathcal{L}\Big{[}\mathcal{L}S+(|A|^2+\frac{1}{2})S\Big{]}+(|A|^2+\frac{1}{2})\Big{[}\mathcal{L}S+(|A|^2+\frac{1}{2})S\Big{]}\Big{]}d\tilde{\mu}\\
        &+\int S^2(2A_{ij}\nabla_i\nabla_jS+2S\text{tr}(A^3))d\tilde{\mu}\\
        &+\int S\Big{[}2SA_{ij}\nabla_i\nabla_jS+2A^{jk}\nabla_jS\nabla_kS-H|\nabla S|^2+Sg^{kl}\nabla_lS\nabla_kH\Big{]}d\tilde{\mu}.
    \end{split}
\end{equation*}
The self-adjointness of $\mathcal{L}$ allows us to write
$$\int S\Big{(}\mathcal{L}\Big{[}\mathcal{L}S+(|A|^2+\frac{1}{2})S\Big{]}+(|A|^2+\frac{1}{2})\Big{[}\mathcal{L}S+(|A|^2+\frac{1}{2})S\Big{]}\Big{)}d\tilde{\mu}=\int \Big{[}\mathcal{L}S+(|A|^2+\frac{1}{2})S\Big{]}^2d\tilde{\mu}.$$
Combining our calculations to this point gives
\begin{equation*}
    \begin{split}
        \Ddot{E}(t)=&4\int \Big{[}\mathcal{L}S+(|A|^2+\frac{1}{2})S\Big{]}^2d\tilde{\mu}\\
        &+2\int S\Big{[}4SA_{ij}\nabla_i\nabla_jS+Sg^{kl}\nabla_lS\nabla_kH+2S^2\text{tr}(A^3)\Big{]}d\tilde{\mu}\\
        &+2\int S\Big{[}2A^{jk}\nabla_jS\nabla_kS-H|\nabla S|^2\Big{]}d\tilde{\mu}\\
        &+\int S^6d\tilde{\mu}-6\int S^3\Big{[}\mathcal{L}S+(|A|^2+\frac{1}{2})S\Big{]}d\tilde{\mu}\\
        =&\int S^6d\tilde{\mu}-6\int S^3\Big{[}\mathcal{L}S+(|A|^2+\frac{1}{2})S\Big{]}d\tilde{\mu}+4\int \Big{[}\mathcal{L}S+(|A|^2+\frac{1}{2})S\Big{]}^2d\tilde{\mu}\\
        &+2\int S\Big{[}4SA_{ij}\nabla_i\nabla_jS+Sg^{kl}\nabla_lS\nabla_kH+2S^2\text{tr}(A^3)\Big{]}d\tilde{\mu}\\
        &+2\int S\Big{[}2A^{jk}\nabla_jS\nabla_kS-H|\nabla S|^2\Big{]}d\tilde{\mu}\\
        =&\int \Big{[}-S^3+2\Big{(}\mathcal{L}S+(|A|^2+\frac{1}{2})S\Big{)}\Big{]}^2d\tilde{\mu}\\
        &-2\int S^3\Big{[}\mathcal{L}S+(|A|^2+\frac{1}{2})S\Big{]}d\tilde{\mu}+2\int S\Big{[}4SA_{ij}\nabla_i\nabla_jS+Sg^{kl}\nabla_lS\nabla_kH+2S^2\text{tr}(A^3)\Big{]}d\tilde{\mu}\\
        &+2\int S\Big{[}2A^{jk}\nabla_jS\nabla_kS-H|\nabla S|^2\Big{]}d\tilde{\mu}
    \end{split}
\end{equation*}
which is \eqref{eq-energy-der2}. $\blacksquare$

In the next section we will define the Dirichlet-Einstein quotient and explore its significance for the long term behavior of our energy functional.

\section{Dirichlet-Einstein Quotient}\label{sec-dirichlet-quotient}
In this section we introduce the Dirichlet-Einstein quotient and analyze its evolution in time in an attempt to better understand the long term behavior of \eqref{eq-normal-mcf}.  In \ref{sec-intro} we have motivated the definition of the Dirichelt-Einstein quotient, which we recall is defined as follows: for a family of hypersurfaces moving by the normal rescaled mean curvatire flow \eqref{eq-normal-mcf} or the recaled mean curvature flow \eqref{eq-mcf-rescaled} (which is not a shrinker), we define the \emph{Dirichlet-Einstein quotient} $N(t)$ by the following formula
\begin{equation}\label{eq-dirichlet-quotient}
    N(t)=\frac{\Dot{E}(t)}{E(t)}=\frac{d}{dt}\text{ln}(E(t)).
\end{equation}

%
%
\subsection{Evolution of the Dirichlet-Einstein Quotient}
In this subsection we investigate the time-evolution of the Dirichlet-Einstein quotient. We prove the following lemma.

\begin{lemma}\label{lemma-N-der}
For a family of hypersurfaces $M_t\subset\mathbb{R}^{n+1}$ moving by the normal rescaled mean curvature flow \eqref{eq-normal-mcf}, N(t) which is defined by \eqref{eq-dirichlet-quotient} evolves by
\begin{equation}\label{eq-N-der}
    \begin{split}
        \Dot{N}(t)=&\\
        &\frac{\int S^2d\tilde{\mu}\int \Big{[}-S^3+2\Big{(}\mathcal{L}S+(|A|^2+\frac{1}{2})S\Big{)}\Big{]}^2d\tilde{\mu}-\Big{[}\int S\Big{[}-S^3+2\Big{(}\mathcal{L}S+(|A|^2+\frac{1}{2})S\Big{)}\Big{]}d\tilde{\mu}\Big{]}^2}{E(t)^2}\\ +\frac{1}{E(t)}\Big{[}&-2\int S^3\Big{(}\mathcal{L}S+(|A|^2+\frac{1}{2})S\Big{)}d\tilde{\mu}+2\int S\Big{(}4SA_{ij}\nabla_i\nabla_jS+Sg^{kl}\nabla_lS\nabla_kH+2S^2\text{tr}(A^3)\Big{)}d\tilde{\mu}\\
        &+\int S^2\Big{(}\langle\nabla H,\nabla S\rangle+H\mathcal{L}S\Big{)}d\tilde{\mu}-2\int S^2\Big{(}\nabla_jA^{jk}\nabla_kS+A^{jk}\nabla_j\nabla_kS-2\frac{x_j}{2}A^{jk}\nabla_kS\Big{)}d\tilde{\mu}\Big{]}
    \end{split}
\end{equation}
for all $t$.
\end{lemma}
\emph{Proof:} We begin by the obvious calculation
$$\Dot{N}=\frac{\Ddot{E}E-\Dot{E}^2}{E^2}$$
which can be computed directly from the definition of $N(t)$. 

This calculation allows us to appeal to equations \eqref{eq-energy-ev} and \eqref{eq-energy-der2} to calculate that
\begin{equation*}
    \begin{split}
        \Dot{N}(t)=\frac{1}{E^2(t)}&\Big{[}\int S^2 d\tilde{\mu}\int \Big{[}-S^3+2\Big{(}\mathcal{L}S+(|A|^2+\frac{1}{2})S\Big{)}\Big{]}^2d\tilde{\mu}\\
        &-2\int S^2 d\tilde{\mu}\int S^3\Big{[}\mathcal{L}S+(|A|^2+\frac{1}{2})S\Big{]}d\tilde{\mu}\\
        &+2\int S^2 d\tilde{\mu}\int S\Big{[}4SA_{ij}\nabla_i\nabla_jS+Sg^{kl}\nabla_lS\nabla_kH+2S^2\text{tr}(A^3)\Big{]}d\tilde{\mu}\\
        &+2\int S^2 d\tilde{\mu}\int S\Big{[}2A^{jk}\nabla_jS\nabla_kS-H|\nabla S|^2\Big{]}d\tilde{\mu}\\
        &-\Big{(}\int S^4d\tilde{\mu}+2\int S\Big{[}\mathcal{L}S+(|A|^2+\frac{1}{2})S\Big{]}d\tilde{\mu}\Big{)}^2\Big{]}.
    \end{split}
\end{equation*}
Now we shall investigate some of these terms. Consider the quantity
$$\int S\Big{[}2A^{jk}\nabla_jS\nabla_kS-H|\nabla S|^2\Big{]}d\tilde{\mu}.$$
We may integrate this by parts, term by term as follows. We have
\begin{equation*}
    \begin{split}
        2\int A^{jk}S\nabla_jS\nabla_kSd\tilde{\mu}=&\int A^{jk}\nabla_jS^2\nabla_kSd\tilde{\mu}\\
        =&\int \nabla_jS^2A^{jk}\nabla_kSe^{-\frac{
        |x|^2}{4}}d\mu\\
        =&-\int S^2\Big{[}\nabla_jA^{jk}\nabla_kSe^{-\frac{
        |x|^2}{4}}+A^{jk}\nabla_j\nabla_kSe^{-\frac{
        |x|^2}{4}}+A^{jk}\nabla_kS(-\frac{x_j}{2})e^{-\frac{ |x|^2}{4}}\Big{]}d\mu\\
        =&-\int S^2\Big{[}\nabla_jA^{jk}\nabla_kS+A^{jk}\nabla_j\nabla_kS-\frac{x_j}{2}A^{jk}\nabla_kS\Big{]}d\tilde{\mu}.
    \end{split}
\end{equation*}
Next, we write
\begin{equation*}
    \begin{split}
        -\int HS|\nabla S|^2d\tilde{\mu}=&-\frac{1}{2}\int\nabla_jS^2\big{(}H\nabla_jSe^{-\frac{ |x|^2}{4}}\big{)}d\mu\\
        =&\frac{1}{2}\int S^2\Big{[}\langle\nabla H,\nabla S\rangle e^{-\frac{ |x|^2}{4}}+H\Delta Se^{-\frac{ |x|^2}{4}}-H\frac{x_j}{2}\nabla_jSe^{-\frac{ |x|^2}{4}}\Big{]}d\mu\\
        =&\frac{1}{2}\int S^2\Big{[}\langle\nabla H,\nabla S\rangle+H\mathcal{L}S\Big{]}d\tilde{\mu}.
    \end{split}
\end{equation*}
Combining these calculations gives
\begin{equation*}
    \begin{split}
        \Dot{N}(t)=\frac{1}{E^2(t)}&\Big{[}\int S^2 d\tilde{\mu}\int \Big{[}-S^3+2\Big{(}\mathcal{L}S+(|A|^2+\frac{1}{2})S\Big{)}\Big{]}^2d\tilde{\mu}\\
        &-2\int S^2 d\tilde{\mu}\int S^3\Big{[}\mathcal{L}S+(|A|^2+\frac{1}{2})S\Big{]}d\tilde{\mu}\\
        &+2\int S^2 d\tilde{\mu}\int S\Big{[}4SA_{ij}\nabla_i\nabla_jS+Sg^{kl}\nabla_lS\nabla_kH+2S^2\text{tr}(A^3)\Big{]}d\tilde{\mu}\\
        &-2\int S^2 d\tilde{\mu}\int S^2\Big{[}\nabla_jA^{jk}\nabla_kS+A^{jk}\nabla_j\nabla_kS-\frac{x_j}{2}A^{jk}\nabla_kS\Big{]}d\tilde{\mu}\\
        &+\int S^2 d\tilde{\mu}\int S^2\Big{[}\langle\nabla H,\nabla S\rangle+H\mathcal{L}S\Big{]}d\tilde{\mu}\\
        &-\Big{(}\int S^4d\tilde{\mu}+2\int S\Big{[}\mathcal{L}S+(|A|^2+\frac{1}{2})S\Big{]}d\tilde{\mu}\Big{)}^2\Big{]}
    \end{split}
\end{equation*}
which is \eqref{eq-N-der}. $\blacksquare$
\\
\newline
With the evolution of the Dirichlet-Einstein quotient computed, we can now begin to work on bounding it in accordance with our plan laid out in the introduction (section \ref{sec-intro}).

\subsection{Order of Terms in $\Dot{N}$}
In this subsection we prove a bound on $\Dot{N}$ that will be central to our proof. We will bound $\Dot{N}$ in terms of the various covariant derivatives $\nabla^lS$ of $S$. As we approach a shrinker, since $S\equiv 0$ on a shrinker, we obviously have that $\nabla^lS\to 0$ on $M_t$. The hope is that later on we can prove some bound of the form
\begin{equation*}
    \int_{T_0}^\infty||\nabla^{(l)}S(p,t)||_{\infty, M_t}\leq C_l(\Sigma,\int_{M_0} e^{-\frac{|x|^2}{4}}d\mu-\int_\Sigma e^{-\frac{|x|^2}{4}}d\mu ).
\end{equation*}

Such bounds would enable us to control
$$N(t)=N(T_0)+\int_{T_0}^t\Dot{N}(s)ds$$
as $t\to\infty$.

From our work above the following corollary is immediate.
\begin{corollary}
For a family of hypersurfaces $M_t\subset\mathbb{R}^{n+1}$ moving by the normal rescaled mean curvature flow \eqref{eq-normal-mcf}, if $N(t)$ is the Dirichlet-Einstein quotient, we have that
\begin{equation}\label{eq-N-der-bound}
        \Dot{N}(t)\geq-C(||S||_\infty+||\nabla S||_\infty+||\nabla^2S||_\infty)
\end{equation}
where $C=C(S, |A|,|\nabla A|, H, \nabla H, x)$.
\end{corollary}
\emph{Proof:} Consider term by term on the right hand side of  equation \eqref{eq-N-der}.

We have
$$-\frac{1}{E(t)}\int S^3\Big{(}\mathcal{L}S+(|A|^2+\frac{1}{2})S\Big{)}d\tilde{\mu}\geq-\Big{(}\|S\mathcal{L}S\|_\infty+\||A|^2+\frac{1}{2}\|_\infty\|S^2\|_\infty\Big{)}$$
so this term is suitably bounded by the uniform boundedness of $|A|^2$ and $|\nabla^lS|$ along the flow.

Next we have that
\begin{equation*}\begin{split}\frac{1}{E(t)}\int S&\Big{(}4SA_{ij}\nabla_i\nabla_jS+Sg^{kl}\nabla_lS\nabla_kH+2S^2\text{tr}(A^3)\Big{)}d\tilde{\mu}\geq\\
&-\Big{(}4(\|A\|_{\infty} \|\nabla^2S\|_\infty+\|\nabla H\|_{\infty} \|\nabla S\|_\infty+2\|A^3\|_{\infty} \|S^2\|_\infty\Big{)}\end{split}\end{equation*}
which again shows this term gives us an estimate that is claimes in \eqref{eq-N-der-bound}.
We also  have that
$$\frac{1}{E(t)}\int S^2\Big{(}\langle\nabla H,\nabla S\rangle+H\mathcal{L}S\Big{)}d\tilde{\mu}\geq-\Big{(}\|\nabla H\|_\infty\|\nabla S\|_\infty+\|H\|_\infty (\|\frac{x}{2}\|_\infty\|\nabla S\|_\infty+\|\nabla^2S\|_\infty)\Big{)}$$
which is suitably bounded as $H$ is another geometric quantity over which we have uniform control as $t\to\infty$, as well as the fact that $x$ stays bounded along the flow. Furthermore,
\begin{equation*}\begin{split}-\frac{1}{E(t)}\int &S^2\Big{(}\nabla_jA^{jk}\nabla_kS+A^{jk}\nabla_j\nabla_kS-2\frac{x_j}{2}A^{jk}\nabla_kS\Big{)}d\tilde{\mu}\geq\\
&-\Big{(}\|\nabla A\|_{\infty} \|\nabla S\|_{\infty}+\|A\|_{\infty} \|\nabla^2S\|_\infty+\|x\|_{\infty}\| A\|_{\infty} \|\nabla S\|_\infty\Big{)}\end{split}\end{equation*}

which again is suitably bounded for all the reasons we have discussed prior.

Finally, we use the Cauchy-Schwartz inequality with $S$ and $-S^3+2\Big{(}\mathcal{L}S+(|A|^2+\frac{1}{2})S\Big{)}$ to show that
$$\int S^2d\tilde{\mu}\int \Big{[}-S^3+2\Big{(}\mathcal{L}S+(|A|^2+\frac{1}{2})S\Big{)}\Big{]}^2d\tilde{\mu}-\Big{[}\int S\Big{[}-S^3+2\Big{(}\mathcal{L}S+(|A|^2+\frac{1}{2})S\Big{)}\Big{]}d\tilde{\mu}\Big{]}^2\geq 0.$$

Combining the results from above yields
$$\Dot{N}(t)\geq-C(S, |A|, |\nabla A|, H, \nabla H, x)(||S||_\infty+||\nabla S||_\infty+||\nabla^2S||_\infty)$$
which proves \eqref{eq-N-der-bound}.
 $\blacksquare$

%

\section{Controlling $\nabla^lS$}
In order to use \eqref{eq-N-der-bound} we need to get a better control of $S$ and its covariant derivatives. We take care of that in the following Proposition.
\begin{prop}
Let $\Sigma$ be a compact shrinker. Then if $M_t$ is some family of hypersurface moving by the normal rescaled mean curvature flow \eqref{eq-normal-mcf} satisfying
$$\Omega(M_t)\geq\Omega(\Sigma)$$
for all $t\geq 0$, then we may find $\delta$ and $T_0$ such that if 
\begin{equation}
\label{eq-distance}
dist(M_0, \Sigma) < \delta,
\end{equation}
then we may write $M_t$ as a normal graph of $v(p,t)$ over $\Sigma$ for $t\geq T_0$ and such that for every $l\geq 0$,
\begin{equation}\label{eq-est-S}
    \int_{T_0}^\infty||\nabla^{(l)}S(p,t)||_{\infty, M_t}\, dt\leq C_l(\Sigma,\int_{M_0} e^{-\frac{|x|^2}{4}}d\mu-\int_\Sigma e^{-\frac{|x|^2}{4}}d\mu ).
\end{equation}
\end{prop}

\emph{Proof:} First, suppose that $\varphi:M\times[0,\infty)\to\mathbb{R}^{n+1}$ is a normal rescaled mean curvature flow moving by \eqref{eq-normal-mcf}. Then we may find a family of diffeomorphisms $\psi(t):M\to M$ such that we have for every $t$ that
\begin{equation*}
    \frac{d}{dt}\varphi(\psi(t)(p),t)=H(\psi(t)(p),t)\nu(\psi(t)(p),t)+\frac{\varphi(\psi(t)(p),t)}{2}
\end{equation*}
in other words, that $\tilde{\varphi}(p,t)=\varphi(\psi(t)(p),t)$ moves by the (original) rescaled mean curvature flow \eqref{eq-mcf-rescaled}. Let us write $\tilde{M}_t$ for this family of hypersurfaces moving by \eqref{eq-mcf-rescaled}, keeping in mind that for every $t$, we have $\tilde{M}_t=\psi(t)(M_t)$,
that is, our underlying geometric object is the same.
Since we are moving by \eqref{eq-mcf-rescaled}, if we define the new variable $\tau=-e^{-t}$, 
and $\Bar{\varphi}:M\times[-1,0)\to\mathbb{R}^{n+1}$ by the equation
$$\Bar{\varphi}(\cdot,\tau)=e^{-\frac{t}{2}}\tilde{\varphi}(\cdot,t),$$
then $\Bar{\varphi}(M_\tau)$ moves by the mean curvature flow \eqref{eq-mcf}. We write $\Bar{M}_\tau$ for this family of hypersurfaces, and have that $\Bar{M}_{-1}=\tilde{M}_0=M_0$, as hypersurfaces. Note also that  $\Sigma$  is a stationary solution to \eqref{eq-normal-mcf}, and that the mean curvature flow $\Sigma_\tau  =\sqrt{-\tau}\Sigma$, up to tangential diffeomorphisms satisfies \eqref{eq-mcf}. 
%
%
By our assumption we have that
$$\Omega(M_0)\geq\Omega(\Sigma),$$
and hence by the continuity of the dependence of the mean curvature flow on its initial conditions, we may find $\tau_2<0$, and some $\delta_0$ such that if $dist(M_0,\Sigma)<\delta_0$,
then the flow $(\Bar{M}_\tau)_{\tau\in(-1,\tau_2)}$ is sufficiently close in measure to the flow $\Sigma_\tau$. 
Thus, by Theorem 1.1 in  \cite{Sc} we have  the solutions $\tilde{M_t}$ of the rescaled flow \eqref{eq-mcf-rescaled}, may be written as normal graphs of a function $v$ over $\Sigma$ for $\frac{-1+\tau_2}{2}<\tau<0$. Hence we may write $\tilde{M}_t$ as a normal graph of a function $v(p,t)$ over $\Sigma$ for all $t>-\text{log}(-\frac{-1+\tau_2}{2})=:T_0$.

Again, by the continuous dependence, we may also choose $\delta_0$ small enough that 
$$\|v(T_0)\|_{C^{2,\alpha}}<\sigma_0$$
where $\sigma_0$ and  $\theta\in(\frac{1}{2},1)$ are as in \cite{Sc}, such that if $v$ is a $C^{2,\alpha}$ section of $T^\perp\Sigma$ with $\|v\|_{C^{2,\alpha}}<\sigma_0$ then

\begin{equation}\label{eq-Loj}||\nabla \Omega(v)||_{L^2(\Sigma)}\geq|\Omega(v)-\Omega(0)|^\theta.\end{equation}

Now, for $t>T_0$, we consider the quantity $Q(t)$ defined by
$$\tilde{Q}(t)=\Omega(\tilde{M}_t)-\Omega(\Sigma).$$
Obviously, by Huisken's monotonicity formula, $Q(t)$ is monotone decreasing, and $Q(t)\geq 0$. This together with monotonicity of $Q(t)$ implies that if $Q(t_0)=0$, we must have $Q(t)\equiv 0$ for all $t\geq t_0$, meaning, of course, that
$$\frac{d}{dt}\Omega(\tilde{M}_t)\equiv 0$$
for $t\geq t_0$, which would mean that $\tilde{M}_t$ is a shrinker for $t\geq t_0$ by Huisken's monotonicity formula. In fact we may apply the backwards uniqueness results in \cite{HH} to show that $\tilde{M}_t$ is a shrinker for all times. Obviously, in this case, we have \eqref{eq-est-S} for free since $\nabla^lS\equiv 0$
on any shrinker. 

As such, we will now consider the case when $Q(t)>0$ for all $t>T_0$. We proceed similarly to \cite{Ko}. Fixing $\beta\in(2-\frac{1}{\theta}, 1)$, gives that $0<(2-\beta)\theta<1$, so we calculate directly that
$$-\frac{d}{dt}\tilde Q^{1-(2-\beta)\theta}=(1-(2-\beta)\theta)\tilde Q^{-(2-\beta)\theta}\frac{d\tilde Q}{dt}$$
and Huisken's monotonicity formula gives
$$-\frac{d}{dt}\tilde Q^{1-(2-\beta)\theta}=(1-(2-\beta)\theta)||S||^2_{L^2(e^{-\frac{|x|^2}{4}}d\mu_{M_t})}\tilde Q^{-(2-\beta)\theta}.$$
The Lojasiewicz inequality \eqref{eq-Loj} allows us to write that
$$\tilde Q^{-(2-\beta)\theta}||S||^{2-\beta}_{L^2(e^{-\frac{|x|^2}{4}}d\mu_{M_t})}=\tilde Q^{-(2-\beta)\theta}||\nabla\Omega(M_t)||^{2-\beta}_{L^2(d\mu_{M_t})}\geq 1,$$
which implies that
$$-\frac{d}{dt}\tilde Q^{1-(2-\beta)\theta}\geq C_0^{-1}||S||^\beta_{L^2(e^{-\frac{|x|^2}{4}}d\mu_{M_t})}$$
for some $C_0=C_0(\beta,\theta)$.

At this point we should note that every calculation we have performed for $\tilde Q(t)$, assuming that our family of hypersurfaces is moving by the rescaled mean curvature flow (2) also holds true for $Q(t)$ if we define $Q(t)$ as 
$$Q(t)=\Omega(M_t)-\Omega(\Sigma)$$
where $M_t$ is the corresponding family of hypersurfaces moving via the normal rescaled mean curvature flow \eqref{eq-normal-mcf}, since
$$\Omega(M_t)=\Omega(\psi(t)(M_t))$$
for all $t$ for any family $\psi$ of tangential diffeomorphisms. As such, from here on we simply use $Q(t)$ and $\tilde{Q}(t)$ interchangeably, writing $Q(t)$ for either.

Because of this, we can write that for $0<a,s<\infty$,
\begin{equation}\label{eq-S-beta}
    \int_a^s||\partial_t \varphi (p,t)||^\beta_{L^2(e^{-\frac{|x|^2}{4}}d\mu_{M_t})}dt\leq C(C_0,\beta,\theta,C_L)\Big{[}Q(a)^\gamma-Q(s)^\gamma\Big{]}
\end{equation}
for $\gamma=1-(2-\beta)\theta$.

Derivative estimates for the rescaled mean curvature flow (2), are obviously invariant under tangential diffeomorphism and allow us to write
$$\|\nabla^lS\|_{\infty,M_t}\leq K_l$$
on $M_t$ uniformly. We will now explore the implications of inequality \eqref{eq-S-beta}.  \\

%

Returning now to the matter at hand, we shall break the proof into two cases depending on the relationship between $\beta$ and $1-\beta$.\\
\emph{Case 1:} $\beta<1-\beta$

In this case, given any $p$, we may choose $N=N(p)$ such that 
$$\frac{p}{N}\leq\beta$$
and therefore, of course, $1-\frac{p}{N}\geq 1-\beta$. We may use interpolation inequalities to write that for $0\leq q\leq p$ we have
\begin{equation}
\label{eq-interpolation}
\int |\nabla^q(S)|^2d\mu\leq C\Big{[}\int|\nabla^NS|^2d\mu\Big{]}^{\frac{q}{N}}\Big{[}\int|S|^2d\mu\Big{]}^{1-\frac{q}{N}}
\end{equation}
or
$$\|\nabla^qS\|\leq C\|\nabla^NS\|^{\frac{q}{N}}\|S\|^{1-\frac{q}{N}}.$$

In what follows we consider the quantity 
$$s:=1-\frac{q}{N}-(1-\beta)=\beta-\frac{q}{N},$$
which by our choice of $N$ guarantees that $s\geq 0$.

We obviously have, by the definition of $s$,
$$\|S\|^{1-\frac{q}{N}}=\|S\|^{1-\beta}\|S\|^{s}.$$
Since $\|S\|\leq\|S\|_{L^{2,N}}$, $\|\nabla^NS\|\leq\|S\|_{L^{2,N}}$, and $s+\frac{q}{N}=\beta$ we can write 
$$\|\nabla^qS\|\leq C\|S\|^{1-\beta}\|S\|^{\beta}_{L^{2,N}}$$
for every $q\leq p$. This allows us the find $C'(p,m)$ such that we may write
$$\|S\|_{L^{2,p}}\leq C'\|S\|^{1-\beta}\|S\|^\beta_{L^{2,N}}.$$
Since $\beta<1-\beta$ and $\|S\|\leq\|S\|_{L^{2,N}}$, we can write that
$$\|S\|_{L^{2,p}}\leq C'\|S\|^{\beta}\|S\|^{1-\beta}_{L^{2,N}}.$$\\
{\color{red} I still have a problem with above, how do you use $\beta < 1-\beta$? This is not true the way it is written. I do not know how you intend to use it later. }
\emph{Case 2:} $1-\beta\leq\beta$

In this case, given any $p$, we may choose $N=N(p)$ such that 
$$\frac{p}{N}\leq 1-\beta$$
and therefore, of course, $1-\frac{p}{N}\geq \beta$. 


In what follows we consider the quantity 
$$s:=1-\frac{q}{N}-\beta=(1-\beta)-\frac{q}{N},$$
which by our choice of $N$ guarantees that $s\geq 0$.

We obviously have, by the definition of $s$,
$$\|S\|^{1-\frac{q}{N}}=\|S\|^{\beta}\|S\|^{s}.$$
Since $\|S\|\leq\|S\|_{L^{2,N}}$, $\|\nabla^NS\|\leq\|S\|_{L^{2,N}}$, and $s+\frac{q}{N}=1-\beta$, by the interpolation inequality \eqref{eq-interpolation}, similarly as in the previous case we have 
$$\|\nabla^qS\| \leq C\|S\|^{\beta}\|S\|^{1-\beta}_{L^{2,N}}$$
for every $q\leq p$. This allows us to find $C'(p,m)$ such that we may write
$$\|S\|_{L^{2,p}}\leq C'\|S\|^{\beta}\|S\|^{1-\beta}_{L^{2,N}}.$$

Thus either way, since $|\nabla^lS|$ is uniformly bounded for every $l$, and since $M_0$ is compact, we may write that 
$$\|S\|_{L^{2,p}}\leq C''(\beta,p,\Sigma, M_0)\|S\|^{\beta}_{L^2(e^{-\frac{|x|^2}{4}})}.$$\\
The Sobolev embedding theorem says that if
%
%
$n$ is odd, 
we have a natural embedding $W^{k,2}\subset C^{k-\frac{n+1}{2},\frac{1}{2}}$, and if $n$ is even then
we have natural embeddings $W^{k,2}\subset C^{k-1-\frac{n}{2},\gamma}$,
for every $\gamma\in(0,1)$.
Given that we have
\begin{equation*}
    \|S\|_{L^{2,k}}\leq C''(\beta,k,\Sigma, M_0)\|S\|^{\beta}_{L^2(e^{-\frac{|x|^2}{4}})}
\end{equation*}
for any $k$, we may simply enlarge $k$ as much as possible to claim that due to the Sobolev embedding theorems
we may find a $C''$ such that
\begin{equation}
    \|S\|_{C^l}\leq C''(\beta,l, n,\Sigma, M_0)\|S\|^{\beta}_{L^2(e^{-\frac{|x|^2}{4}})}
\end{equation}
for any $l$.

Recall that for $t>T_0$ we may express $M_t$ as the normal graph of a function $v(t)$ over $\Sigma$. 
We have, for $T_0 \le a \le s < \infty$,
$$\|v(s)-v(a)\|_{C^l_\Sigma}\leq C(l)\int_a^s\|\partial_t\varphi(p,t)\|_{C^l_\Sigma}dt$$
and
$$\int_a^s\|\partial_t\varphi(p,t)\|_{C^l_\Sigma}dt\leq C(l,\Sigma,M_0,C'')\int_a^s\|\partial_t\varphi (p,t)\|_{C^l_{M_t}}dt.$$

Recall that $\partial_t\varphi=S\nu$ so 
$$\|S\nu\|_{C^l}\leq\sum_{k=0}^l\Big{(}\sum_{p+q=k}\|\nabla^pS\|_\infty\|\nabla^q\nu\|_\infty\Big{)}\leq C\|S\|_{C^l}$$
where $C$ is finite as it is dependant on the $\nu$ and its covariant derivatives which are dependant on the covariant derivatives of the second fundamental form, which are all bounded along the flow. And since
$$\|S\|_{C^l}=\|S\|_\infty+\|\nabla S\|_\infty+...+\|\nabla^lS\|_\infty,$$
we have that inequality (32) combined with inequality (34) give
\begin{equation}
    \int_a^s\Big{[}\|S(t)\|_\infty+\|\nabla S(t)\|_\infty+...+\|\nabla^lS(t)\|_\infty\Big{]}dt\leq C(C_0,\beta,\theta,C_L, l, \Sigma, M_0)\Big{[}Q(a)^\gamma-Q(s)^\gamma\Big{]}
\end{equation}
and the monotonicity of $Q(t)$ allows us to simply write
\begin{equation}
    \int_a^s\Big{[}\|S(t)\|_\infty+\|\nabla S(t)\|_\infty+...+\|\nabla^lS(t)\|_\infty\Big{]}dt\leq C(C_0,\beta,\theta,C_L, l, \Sigma, M_0)\Big{[}Q(T_0)^\gamma\Big{]}
\end{equation}
for all $a,s\in[T_0,\infty)$ and this, pushing $s\to\infty$ and $a\to T_0$, implies (30), which completes the proof of the proposition. $\blacksquare$

\section{Rate of convergence}
In this section we will prove Theorem \ref{thm-main}.  Our approach to the proof will be through contradiction. We shall first prove the following  lemma, which we refer to as \emph{the rate lemma}, and which concerns the $C^{2,\alpha}$ norm of the graph-function which represents $M_t$ as a graph of $v(t)$ over $\Sigma$.
\begin{lemma}
 If  for every $m\geq 0$, we have $C_m'$ such that for $t_i>T_0$
$$||v(\cdot, t_i)||_{C^{2,\alpha}}< C'_me^{-mt_i},$$
then we may also find constants $C_m''$ such that for $t>T'$ (where $T'$ is the first $t_i>T_0$)

\begin{equation*}\int_{M_{t_i}}e^{-\frac{|x|^2}{4}}d\mu-\int_\Sigma e^{-\frac{|x|^2}{4}}d\mu\leq C^{''}_me^{-mt_i}.\end{equation*}
\end{lemma}
\emph{Proof:} Suppose, in accordance with our hypotheses, that for every $m>0$, we could find $C'_m$ such that for $t>T_0$ we have
$$\|{v}\|_{C^{2,\alpha}}<C'_me^{-mt_i}$$
for all $t_i$.

We now take the time to rewrite $\int_{M_{t_i}}e^{-\frac{|x|^2}{4}}d\mu$, as Schulze does in \cite{Sc}, as an integral over $\Sigma$ as follows
$$\int_{M_{t_i}}e^{-\frac{|x|^2}{4}}d\mu=\int_\Sigma e^{-\frac{|y+v(y,t_i)\,\nu|^2}{4}}|J|(y,v,\nabla^\Sigma v)d\mu(y)$$
where $|J|(y,v,\nabla^\Sigma v)$ is the Jacobian of the transformation $y\to y+v(y,t_i)\nu$. We shall now rewrite $\int_{M_{t_i}}e^{-\frac{|x|^2}{4}}d\mu-\int_\Sigma e^{-\frac{|x|^2}{4}}d\mu$ as 
$$\int_\Sigma \Big{(}e^{-\frac{|y+v(y,t_i)\nu|^2}{4}}|J|(y,v,\nabla^\Sigma v)-e^{-\frac{|y|^2}{4}}\Big{)}d\mu(y).$$
We further simplify this expression by writing
$$e^{-\frac{|y+v(y,t_i)\nu|^2}{4}}=e^{-\frac{|y|^2}{4}}\cdot e^{-\frac{|v(y,t_i)
\nu|^2}{4}}\cdot e^{-\frac{\langle y, v(y,t_i)\nu\rangle}{2}}$$
and then we rewrite our key quantity $\int_\Sigma \Big{(}e^{-\frac{|y+v(y,t_i)\nu|^2}{4}}|J|(y,v,\nabla^\Sigma v)-e^{-\frac{|y|^2}{4}}\Big{)}d\mu(y)$, breaking it up as follows, writing our Jacobian as simply $|J|$
\begin{equation*}
    \begin{split}
        \int_\Sigma \Big{(}e^{-\frac{|y+v(y,t_i)\nu|^2}{4}}|J|(y,v,\nabla^\Sigma v)-e^{-\frac{|y|^2}{4}}\Big{)}d\mu(y)=&\int_\Sigma \Big{(}e^{-\frac{|y+v(y,t_i)\nu|^2}{4}}|J|-e^{-\frac{|y|^2}{4}}\cdot e^{-\frac{|v(y,t_i)\nu|^2}{4}}|J|\Big{)}d\mu(y)\\
        &+\int_\Sigma \Big{(}e^{-\frac{|y|^2}{4}}\cdot e^{-\frac{|v(y,t_i)\nu|^2}{4}}|J|-e^{-\frac{|y|^2}{4}}|J|\Big{)}d\mu(y)\\
        &+\int_\Sigma \Big{(}e^{-\frac{|y|^2}{4}}|J|-e^{-\frac{|y|^2}{4}}\Big{)}d\mu(y)\\
        =&\int_\Sigma \Big{(}e^{-\frac{\langle y, v(y,t_i)\nu\rangle}{2}}-1\Big{)}e^{-\frac{|y|^2}{4}}\cdot e^{-\frac{|v(y,t_i)\nu|^2}{4}}|J|d\mu(y)\\
        &+\int_\Sigma \Big{(} e^{-\frac{|v(y,t_i)\nu|^2}{4}}-1\Big{)}e^{-\frac{|y|^2}{4}}|J|d\mu(y)\\
        &+\int_\Sigma \Big{(}|J|-1\Big{)}e^{-\frac{|y|^2}{4}}d\mu(y)\\
        =&:A+B+D.
    \end{split}
\end{equation*}
Since we assumed our flow is compact, we know that $y$ is bounded, and using that $v(y,t_i)$ is small, there exists a uniform constant $C$ that may change from line to line in a uniform way, so that
\begin{equation*}
    \begin{split}
        A=&\int_\Sigma \Big{(}e^{-\frac{\langle y, v(y,t_i)\nu\rangle}{2}}-1\Big{)}e^{-\frac{|y|^2}{4}}\cdot e^{-\frac{|v(y,t_i)\nu|^2}{4}}|J|d\mu(y)\\
        \leq&\int_\Sigma \Big{|}\langle y, v(y,t_i)\nu\rangle\Big{|}e^{-\frac{|y|^2}{4}}\cdot e^{-\frac{|v(y,t_i)\nu|^2}{4}}|J|d\mu(y)\\
        \leq&C\|v\|_{C_0}\int_\Sigma e^{-\frac{|y|^2}{4}}d\mu(y)\\
    \end{split}
\end{equation*}
since $|J|$ is bounded. We also have a uniform constant $C$ which may change uniformly from line to line such that
\begin{equation*}
    \begin{split}
        B=&\int_\Sigma \Big{(} e^{-\frac{|v(y,t_i)|^2}{4}}-1\Big{)}e^{-\frac{|y|^2}{4}}|J|d\mu(y)\\
        \leq&\int_\Sigma \frac{|v(y,t_i)|^2}{2}e^{-\frac{|y|^2}{4}}|J|d\mu(y)\\
        \leq& C\|v\|_{C_0}^2\int_\Sigma e^{-\frac{|y|^2}{4}}d\mu(y)\\
        \leq& C\|v\|_{C_0}\int_\Sigma e^{-\frac{|y|^2}{4}}d\mu(y)
    \end{split}
\end{equation*}
where we used the Talor expansion of the exponential function and the smallness of $v$. Finally, we have $D$. Our Jacobian $|J|$ is the Jacobian of the transformation
$y\to y+v(y,t_i)\nu(y)$, with the entries of its corresponding matrix $J$ given by
$$J_{jk}=\partial_j(y+v(y,t_i)\nu(y))_k=\delta_{jk}+\partial_jv(y,t_i)\nu_k-v(y,t_i)h_{jl}\big{(}\frac{\partial\varphi}{\partial x_l}\big{)}_k.$$
One can see that every term of $|J|-1$ contains $v$ or its first derivatives (or the entries of the second fundamental form which are of course bounded),and so it is also only dependant on $\|{v}_{C^1}$, and we can see, by writing out the entries of its corresponding matrix, that $J-1$ depends on $v$ and its first derivatives in a sub-linear way, as such, we can write. 
$$\int_{M_{t_i}}e^{-\frac{|x|^2}{4}}d\mu-\int_\Sigma e^{-\frac{|x|^2}{4}}d\mu\leq \tilde{C}\|v\|_{C^{2,\alpha}}$$
where $\tilde{C}$ is a uniform constant dependant only on quantities bounded along the flow. $\blacksquare$\\

We are now ready to continue to the proof of Theorem \ref{thm-main}.
\begin{proof}[Proof of Theorem \ref{thm-main}]
We begin by stating Theorem 1.1 of \cite{Sc}, which says

Let $\mathcal{M}=(\mu_t)_{t\in(t_1,0)}$ with $t<0$ be an integral $n$-Brakke flow such that 
\begin{enumerate}[i)]
    \item $(\mu_t)_{t\in(t_1,t_2)}$ is sufficiently close in measure to $\mathcal{M}_\Sigma$ for some $t_1<t_2<0$.
    \item $\int_{M_t}\frac{1}{(-4\pi t)^\frac{n}{2}}e^{-\frac{|x|^2}{4}}d\mathcal{H}^n\geq\int_\Sigma\frac{1}{(-4\pi t)^\frac{n}{2}}e^{-\frac{|x|^2}{4}}d\mathcal{H}^n$.
\end{enumerate}
Then $\mathcal{M}$ is a smooth flow for $t\in[(t_1+t_2)/2,0)$, and the rescaled surfaces $\tilde{M}_t$ can be rewritten as normal graphs over $\Sigma$, given by smooth sections $v(t)$ of the normal bundle with $\|v(t)\|_{C^m(T^\perp\Sigma)}$ uniformly bounded for all $t\in[(t_1+t_2)/2,0)$ and all $m\in\mathbb{N}$. Furthermore, there exists a self-similarly shrinking surface $\Sigma'$ with
$$\Sigma'=\text{graph}_\Sigma(v')$$
and
$$\|v(t)-v'\|_{C^m}\leq c_m(\text{log}(-1/t))^{-\alpha_m}$$
for some constants $c_m>0$ and exponents $\alpha_m>0$ for all $m\in\mathbb{N}$.

If $\tilde{\varphi}:M\times[0,\infty)\to\mathbb{R}^{n+1}$ is a rescaled mean curvature flow moving by \eqref{eq-mcf-rescaled}, then as before we define
$\tau=-e^{-t}$, and $\varphi:M\times[-1,0)\to\mathbb{R}^{n+1}$ is moving by \eqref{eq-mcf} via rescaling $\varphi(\cdot,\tau)=e^{-\frac{t}{2}}\widetilde{\varphi}(\cdot, t).$
In the non-rescaled flow, per our hypothesis, we can find some $\tau_1,\tau_2\in[-1,0)$ ($\tau_1<\tau_2$) such that our restriction of the mean curvature flow $\varphi:M\times[\tau_1,\tau_2)\to\mathbb{R}^{n+k}$ is sufficiently close in measure to the homothetically shrinking flow produced by the shrinker $\Sigma$. Thus we can apply Theorem 1.1 from \cite{Sc}, to get that our solutions moving by \eqref{eq-mcf-rescaled} can be expressed as normal graphs of a function $v(\cdot,\tau)$ for $\frac{\tau_1+\tau_2}{2}<\tau<0$.

As such, we may define $v(\cdot,t)$ (we use the same $v$, just change its time variable) for all $t>-\text{log}(-\frac{\tau_1+\tau_2}{2})=:T_0$ such that our original flow $\tilde{\varphi}$ is expressible as a normal graph of $v(\cdot,t)$ over $\Sigma$ for all $t>T_0$. This establishes the needed graphicality.

We will now approach the proof by contradiction. Suppose that the theorem were false, that is, suppose that $M_t$ is a rescaled mean curvature flow such that for every $m>0$, we have $C'_m$ such that for $t_i>T_0$, we have
$$\|{v(\cdot,t_i)}_{C^{2,\alpha}}<C'_me^{-mt_i}$$
for every $i$, but $M_t$ is never a shrinker. 

By Lemma 3.1 equation (23), we have that 
\begin{equation*}
    \begin{split}
        \Dot{E}(t)=&-\int S^4d\tilde{\mu}+2\int S\Big{[}\mathcal{L}S+(|A|^2+\frac{1}{2})S\Big{]}d\tilde{\mu}\\
        =&-\int S^4d\tilde{\mu}-2\int|\nabla S|^2d\tilde{\mu}+2\int(|A|^2+\frac{1}{2})S^2d\tilde{\mu}\\
        \leq&2\int(|A|^2+\frac{1}{2})S^2d\tilde{\mu}\\
        \leq & K(|A|^2)E(t)
    \end{split}
\end{equation*}
so we can write for $t>T_0$ that
$$E(t)\leq E(T_0)e^{K(t-T_0)}.$$

Now we recall the quantity
$$N(t)=\frac{\Dot{E}}{E}=\frac{d}{dt}\text{ln}(E).$$
We claim that we are able to produce some $K'>0$, such that we have that $E(t)\geq E(T_0)e^{-K'(t-T_0)}$. To show this, we will prove that $\text{ln}(E(t))\geq-K'(t-T_0)+\text{ln}(E(T_0))$ by proving that for all $t$, we have
$$N(t)\geq-K'$$
i.e. that $N$ is bounded from below for all time. Corollary 8.1 gives us that 
$$\Dot{N}(t)\geq-C(||S||_\infty+||\nabla S||_\infty+||\nabla^2S||_\infty).$$
We then apply Proposition 5.1, inequality (33) for the cases $l=0,1,2$
to get that for for all $t$,
\begin{equation}
    \begin{split}
        N(t)=&\int_0^t\Dot{N}(\xi)d\xi\\
        \geq&-C\int_0^\infty(\|{S(\xi)}_\infty+\|{\nabla S(\xi)}_\infty+\|{\nabla^2S(\xi)}_\infty)d\xi\\
        =:&K'>-\infty,
    \end{split}
\end{equation}
which gives us the appropriate $K'$, and allows us to write that
\begin{equation}E(T_0)e^{-K'(t-T_0)}\leq E(t)\leq E(T_0)e^{K(t-T_0)}\end{equation}
for all $t>T_0$. Huisken's monotonicity formula gives
$$\frac{d}{dt}\Omega(M_t)=-E(t)$$
so we can use inequality (38) to write that for $s\geq t$
\begin{equation*}
\begin{split}
\Omega(M_t)-\Omega(M_s)=&\int_t^s E(\zeta)d\zeta\\
\geq& \int_t^sE(T_0)e^{-K'(\zeta-T_0)}d\zeta\\
=& \frac{E(T_0)}{K'}\Big{[}e^{-K'(t-T_0)}-e^{-K'(s-T_0)}\Big{]}.
\end{split}
\end{equation*}
Pushing $s\to\infty$ gives
\begin{equation}
\Omega(M_t)-\Omega(\Sigma)\geq \frac{E(T_0)}{K'}e^{-K'(t-T_0)}.
\end{equation}

We apply the rate lemma, Lemma 6.1, to conclude that for every $m>0$, we must also have $C''_m$ such that for $t_i>T_0$, we have
\begin{equation}\int_{M_{t_i}}e^{-\frac{|x|^2}{4}}d\mu-\int_\Sigma e^{-\frac{|x|^2}{4}}d\mu\leq C''_me^{-mt_i}\end{equation}
for every $i$.

Combining inequalities (39) and (40) yields
\begin{equation*}
    E(T_0)\leq Ce^{(K'-m)t_i}
\end{equation*}
for every $m$ and $i$. Since we may increase $m$ and $i$ arbitrarily, we conclude that we must have 
\begin{equation}
    E(T_0)=0,
\end{equation}
so $M_{T_0}$ must be a shrinker. By the backwards uniqueness results for mean curvature flow proven in \cite{HH}, we have that our flow is a shrinker for all time, which is of course a contradiction. We are left to conclude that either there exist $C,m>0$ such that
$
\|{v(\cdot,t_i)}_{C^{2,\alpha}(\Sigma)}\geq Ce^{-mt_i}
$
or $M_t$ is a shrinker moving by diffeomorphisms. $\blacksquare$

\end{proof}

\end{document}